\documentclass{article}
\usepackage{graphicx} 
\usepackage{amssymb}
\usepackage{amsfonts}
\author{Jules Flin$^\dagger$}

\usepackage{amsmath, amsthm, amssymb, amsfonts}
\usepackage{dsfont}

\theoremstyle{plain}
\newtheorem{theorem}{Theorem}[section] 
\newtheorem{proposition}[theorem]{Proposition} 
\newtheorem{lemma}[theorem]{Lemma}
\newtheorem{corollary}[theorem]{Corollary}

\theoremstyle{definition}
\newtheorem{definition}[theorem]{Definition}
\newtheorem{example}[theorem]{Example}

\theoremstyle{remark}
\newtheorem{remark}[theorem]{Remark}

\usepackage[dvipsnames, svgnames, table]{xcolor}

\usepackage[colorlinks=true, linkcolor=blue, citecolor=blue, urlcolor=blue]{hyperref}

\usepackage{authblk}
\title{Stationary Distribution of Brownian Motion in the Half-Plane with Two-sided Reflections}
\author{Jules Flin}
\affil{SAMOVAR, Télécom SudParis, Institut Polytechnique de Paris, 92120 Palaiseau, France}
\date{}

\begin{document}
\maketitle
\begin{abstract}
We investigate the unique stationary measure of a positive recurrent reflecting Brownian motion in the upper half-plane, where the direction of reflection is constant on each half-axis. The Laplace transform of the stationary distribution is characterized by a functional equation, whose resolution is reduced to solving a discontinuous Riemann boundary value problem. By applying the Sokhotski-Plemelj formulas, we derive an explicit expression for the Laplace transform. Finally, we establish the local behavior of the stationary density at the origin and its asymptotics along the boundary axes using Tauberian theorems and asymptotic analysis.
\end{abstract}

\section{Introduction and main results}

The reflected Brownian motion (RBM) in the first quadrant was originally introduced as the heavy-traffic limit for queueing networks by Harrison and Reiman \cite{Harrison_1978,Reiman_1984}. More generally, such processes can be defined within any cone. Through simple linear transformations, all convex cones can be reduced to the quarter-plane, while all non-convex cones are mapped to the three-quarter plane \cite{BoMe,threequarter}. Among the limiting cases excluded from these standard reductions, one finds the wedge of opening $\pi$, which is equivalent to studying an RBM in the upper half-plane. This paper investigates the stationary distribution of such a process.\\

\noindent \textbf{Notations.} \textit{Due to the symmetries of the model, we will frequently employ the symbols $\pm$ and $\mp$. Within a single expression, these should be understood as representing two opposite signs consistently. Furthermore, we use the shorthand notation $\sum_{\pm}a_{\pm} := a_+ + a_-$.}

\subsection*{Reflected Brownian motion in the upper half-plane}
The model at the core of this paper depends on the following set of parameters: a non-singular covariance matrix $\Sigma$ (symmetric, positive definite), two reflection vectors $R_\pm=(r_\pm,1)^\intercal$ (with $r_\pm \in\mathbb R$), and a drift vector $\mu\in\mathbb R^2$. We consider the stochastic process $Z$ by
\begin{equation}\label{eq:trajectoires}
   Z_t=Z_0+\sqrt{\Sigma}W_t+\mu t+\sum_{\pm}R_\pm \ell_\pm(t). 
\end{equation}
Here, $Z_0\in\mathbb R\times \mathbb R_+$ is the initial position, $W$ is a 2-dimensional standard Brownian motion, and $\ell_\pm(t)$ is the local time of $Z$ accumulated up to time $t$ on $\mathbb R_\pm\times \{0\}$, so that the process is confined to the upper half-plane 
$$\mathbb H_+:=\mathbb R\times \mathbb R_+=\{(x,y)\in\mathbb R^2: y\geqslant 0\},$$
and reflected in the direction of $R_\pm$ whenever it hits $\mathbb R_\pm\times \{0\}$.\\

The process $Z$ can be viewed as a Reflected Brownian Motion (RBM) in a wedge with an opening angle $\beta=\pi$. In their full generality, RBMs in orthants are of fundamental importance in queueing theory. They were originally introduced in~\cite{Harrison_1978,Reiman_1984}, and their core properties, such as recurrence and semimartingale property, have since been extensively studied~\cite{Hobson_ROGERS_1993,Varadhan_Williams_1985a,Williams_1985}. The invariant measure of RBMs was investigated by Williams~\cite{Williams_1985b} in the driftless case, and has recently been explicitly computed for opening angles $\beta<\pi$~\cite{BoMe,dreyfus2025degeneratesystemsbrownianparticles,Franceschi_Raschel_2019} and $\beta>\pi$~\cite{threequarter,Fayolle_Franceschi_Raschel_2023} with non-zero drift. These recent works rely either on Tutte's invariant approach or on the resolution of Riemann boundary value problems, drawing a natural analogy with the analytical study of similar discrete models~\cite{livrejaune}. In the present paper, we aim to investigate the limiting case $\beta=\pi$.

\begin{remark}\label{rem:1D}
In our setting, the second coordinate of the process, $(Z_t^{(2)})_{t \ge 0}$, behaves exactly as a one-dimensional Brownian motion reflected at $0$ with drift $\mu_2$. Indeed, the dynamics of this coordinate can be written as:
\begin{equation}
    Z_t^{(2)} = Z_0^{(2)} + W_t^{(2)} + \mu_2 t + \ell_+(t) + \ell_-(t).
\end{equation}
By definition, the local times $\ell_+(t)$ and $\ell_-(t)$ increase only when the process $Z_t$ hits the horizontal axis, which is strictly equivalent to $Z_t^{(2)} = 0$. Consequently, the sum $\ell_+(t) + \ell_-(t)$ corresponds exactly to the local time at $0$ of the one-dimensional process $Z_t^{(2)}$. 
\end{remark}

When $\Sigma$ is the identity matrix, the behavior of this diffusion is largely governed by the parameter $\alpha$, introduced by Varadhan and Williams \cite{Varadhan_Williams_1985a}, defined by
\begin{equation}\label{eq:alpha}
    \alpha:=\frac{\delta_++\delta_--\pi}{\beta},
\end{equation}
where $\delta_\pm\in (0,\pi)$ are the angles depicted in Figure~\ref{fig:model}. As shown in Proposition~\ref{prop:whitening} below, we can assume, without loss of generality, that $\Sigma=\text{I}_2$. Elementary trigonometry then yields
$$\delta_\pm =\frac{\pi}{2}\mp \arctan(r_\pm),$$
which leads to the simplified expression
\begin{equation}\label{eq:alphabis}
    \alpha = \frac{\arctan(r_-) - \arctan(r_+)}{\pi}.
\end{equation}

\begin{figure}
    \centering
    \includegraphics[width=0.65\linewidth]{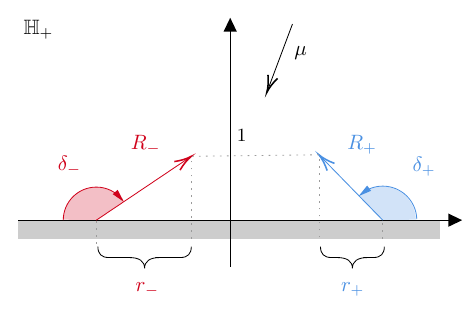}
    \caption{Definition of the angles $\delta_\pm$ through the parameters of the model. Here, $r_->0$ and $r_+<0$.}
    \label{fig:model}
\end{figure}

\begin{remark}[Existence and recurrence] When $\beta=\pi$, it is straightforward to verify that
\begin{equation*}
    \alpha = \frac{\delta_++\delta_-}{\pi}-1 \in (-1,1).
\end{equation*}
According to~\cite{Varadhan_Williams_1985a}, this ensures that the process $Z$ exists as the unique solution to a submartingale problem. Furthermore, it follows from~\cite{Williams_1985} that $Z$ is a semimartingale (often referred to as an SRBM, for Semimartingale Reflected Brownian Motion). Adapting the results from~\cite{Hobson_ROGERS_1993} to the half-plane, one can show that the process is positively recurrent if and only if
\begin{equation}\label{eq:rec}
   \mu_2<0\text{ and } r_+ < \frac{\mu_1}{\mu_2} < r_-.
\end{equation}
In particular, for a fixed pair of reflection vectors, there exists a drift $\mu$ that makes the process positively recurrent if and only if
\begin{equation}\label{eq:valeuralpha}
    \alpha \in (0,1).
\end{equation}
One can refer to~\cite{BoMe} for a comprehensive overview of the properties of the process derived from the value of $\alpha$.
\end{remark}

Under the condition~\eqref{eq:rec}, $Z$ admits a unique stationary distribution $\boldsymbol{\pi}$. We will denote by $\phi$ its Laplace transform, that is
$$\phi(x,y)=\int\!\!\!\!\int_{\mathbb H_+}\exp(xu+yv)\mathrm{d}\boldsymbol{\pi}(u,v).$$
Note that the Laplace transform $\phi(x,y)$ is well-defined and converges absolutely for all $(x,y)\in\mathbb C^2$ such that $\mathfrak{Re}(y)\leqslant 0$ and $\mathfrak{Re}(x)=0$.\\

We conclude this section by showing, as announced, that we can, without loss of generality, restrict ourselves to the case $\Sigma=\mathrm{I}_2$.

\begin{proposition}[Whitening]\label{prop:whitening}
    For any covariance matrix $\Sigma$, there exists a unique orientation-preserving linear transformation $T$ that maps the upper half-plane onto itself and such that $T\Sigma T^\intercal=\mathrm{I}_2$. If $Z$ is a reflected Brownian motion in the upper half-plane with parameters $(\Sigma,\mu,R_\pm)$, then so is $TZ$, with parameters $(\mathrm{I}_2,T\mu,TR_\pm)$. Letting $\boldsymbol{\pi}$ and $\widetilde{\boldsymbol{\pi}}$ denote their respective invariant measures, $\widetilde{\boldsymbol{\pi}}$ is the \textit{pushforward} of $\boldsymbol{\pi}$ by $T$. Explicitly, if $\Sigma=(\sigma_{ij})$, then 
$$T=\frac{1}{\sqrt{\sigma_{22}\det\Sigma}}\left(\begin{array}{cc}
        \sigma_{22} &  -\sigma_{12} \\
         0& \det\Sigma
    \end{array}\right),\quad \pi(u,v)=\frac{1}{\sqrt{\det\Sigma}}\widetilde{\pi}\left(\frac{\sigma_{22}u-\sigma_{12}v}{\sqrt{\sigma_{22}\det\Sigma}},\frac{v}{\sqrt{\sigma_{22}}}\right),$$
where $\pi$ and $\widetilde{\pi}$ are the respective densities of $\boldsymbol{\pi}$ and $\widetilde{\boldsymbol{\pi}}$.
\end{proposition}

\begin{proof}    
In general, there are infinitely many whitening transformations (any whitening followed by a rotation), but the constraint that the upper half-plane must be preserved ensures uniqueness. The remainder of the proof follows from a standard change of variables formula.
\end{proof}

\subsection*{Main results and structure of the paper}

In Section~\ref{sec:bvp}, we establish a functional equation relating $\phi$ and the Laplace transforms of the boundary measures (Proposition~\ref{prop:FE}). This functional equation is easily transformed into a boundary value problem (BVP) on the real line (Proposition~\ref{prop:bvp_disc}). However, the coefficient appearing in this BVP is not continuous on $\overline{\mathbb{R}}=\mathbb{R}\cup\{\pm \infty\}$, which complicates its resolution. Using analytic techniques, we therefore reduce it to a continuous problem (Proposition~\ref{prop:bvp_cont}).

In Section~\ref{sec:sp}, we present the Sokhotski-Plemelj formulas used to solve this type of BVP, and we prove that the functions involved in our problem satisfy the appropriate hypotheses. We then provide an integral formula for the Laplace transforms of the boundary measures (Theorem~\ref{thm:laplace}), from which the bivariate Laplace transform is deduced using the functional equation of Section~\ref{sec:bvp}. For a particular example, we obtain a closed-form expression for the Laplace transform, allowing us to invert it and explicitly recover the density of the invariant measure, see Proposition~\ref{prop:ex}.

Finally, in Section~\ref{sec:asymp}, we apply Tauberian theorems and results from asymptotic analysis to study the behavior of the density of the invariant measure at the origin, and at infinity along the axes of reflection. More precisely, we show that
$$\pi(r\cos \theta,r\sin\theta)\sim C\sin\left(\delta_+-\alpha\theta\right)r^{-\alpha},\quad \text{as }r\to 0^+,$$
for some constant $C$ (see Theorem~\ref{thm:asymp}), and we obtain asymptotics of the form $$\pi(u,0)\asymp |u|^{-\kappa}\exp(-\gamma |u|),\quad \text{as }u\to\pm\infty,$$ (that is, up to a positive multiplicative constant), with explicit exponents $\kappa$ and $\gamma$ (Theorem~\ref{thm:infty}). For the exponent $\kappa$, we identify a phase transition on each boundary with respect to the parameters $r_+$ and $r_-$. This transition occurs at a critical value $r_\pm^\star$ that depends solely on the drift direction $\mu_1/\mu_2$. On either side of this critical parameter, the exponent $\kappa$ is either $0$ or $3/2$, while in the critical regime, $\kappa=1/2$.

\section{Boundary Value Problem(s)}\label{sec:bvp}
As is standard in the study of reflected Brownian motion in a wedge, the starting point is a functional equation often referred to as the \textit{Basic Adjoint Relationship} (BAR), which will be stated in Proposition~\ref{prop:FE}. It involves the \textit{lateral measures} $\boldsymbol{\pi}_\pm$, defined by
$$\boldsymbol{\pi}_\pm(A)=\mathbb{E}_{\boldsymbol{\pi}}\left[\int_0^1 \mathds1_A\left(Z_t^{(1)}\right)\mathrm{d}\ell_\pm(t)\right],$$
through its (univariate) Laplace transforms $\phi_\pm(x)$
$$\phi_\pm(x):=\int_{\mathbb R_\pm}\exp(xu)\mathrm{d}\boldsymbol{\pi}_\pm(u).$$
Note that the Laplace transform $\phi_+$ (\textit{resp.} $\phi_-$) converges for $\mathfrak{Re}(x)\leqslant 0$ (\textit{resp.} $\mathfrak{Re}(x)\geqslant 0$). In particular, $\phi(x,y)$, $\phi_+(x)$ and $\phi_-(x)$ are simultaneously well-defined for $\mathfrak{Re}(x)=0$ and $\mathfrak{Re}(y)\leqslant 0$.
\begin{proposition}[Kernel functional equation]\label{prop:FE} For all $x\in i\mathbb R$ and $y\in\mathbb C$ such that $\mathfrak{Re}(y)\leqslant 0$, the Laplace transforms $\phi$, $\phi_+$ and $\phi_-$ satisfy
\begin{equation}\label{eq:FE}
    K(x,y)\phi(x,y)+\sum_\pm k_\pm(x,y)\phi_\pm(x)=0,
\end{equation}
where 
$$K(x,y)=\frac{1}{2}\left(x^2+y^2\right)+\mu_1 x +\mu_2 y\text{ and }k_\pm(x,y)=r_\pm x+y.$$
\end{proposition}

\begin{proof}
Let $\mathcal{L}$ be the infinitesimal generator of the process in the interior of $\mathbb H_+$, defined for a twice continuously differentiable function $f$ by
$$ \mathcal{L}f(u,v)=\frac{1}{2}\Delta f(u,v) + \mu \cdot \nabla f(u,v).$$
Applying Itô's formula to the test function $f(u,v)=e^{xu+yv}$ evaluated at the semimartingale $Z_t$ yields
$$ f(Z_1) - f(Z_0) = M_1 + \int_0^1 \mathcal{L}f(Z_s) \mathrm{d}s + \sum_{\pm} \int_0^1 \nabla f(Z_s) \cdot R_{\pm} \mathrm{d}\ell_{\pm}(s), $$
where $(M_t)_t$ is a local martingale. Taking the expectation with respect to the stationary distribution $\boldsymbol{\pi}$, the left-hand side and the expectation of $M_1$ vanish. Using the definition of the lateral measures $\boldsymbol{\pi}_{\pm}$, we obtain the equation
$$ \iint_{\mathbb{H}_+} \mathcal{L}f(u,v) \mathrm{d}\boldsymbol{\pi}(u,v) + \sum_{\pm} \int_{\mathbb{R}_{\pm}} \nabla f(u,0) \cdot R_{\pm} \mathrm{d}\boldsymbol{\pi}_{\pm}(u)=0.$$
Observe that $\mathcal{L}f(u,v) = K(x,y)f(u,v)$ and $\nabla f(u,0) \cdot R_{\pm} = k_{\pm}(x,y) e^{xu}$ to obtain the functional equation. See~\cite{Harrison_Williams_1987} for a more detailed proof. 
\end{proof}

\begin{remark}[ADN elliptic system]\label{rem:ADN} It is worth pointing out that the characterization of the stationary measure obtained via our functional equation can be equivalently reformulated in terms of an elliptic Boundary Value Problem. More precisely, one can show that the density $\pi$ of the invariant measure $\boldsymbol{\pi}$ satisfies the following partial differential equation system:
\begin{equation}\label{eq:EBVP_system}
\left\{\begin{array}{ll}
\displaystyle\mathcal{L}^*\pi(u,v):=\left(\frac{1}{2}\Delta-\mu\cdot \nabla\right)\pi(u,v)=0 & \text{for all } (u,v)\in \overset{\circ}{\mathbb{H}}_+, \\[8pt]
\left(\displaystyle\frac{\partial}{\partial v}-r_\pm\frac{\partial}{\partial u}-2\mu_2\right)\pi(u,0)=0&\text{for all }\pm u>0.
\end{array}\right.
\end{equation}
See~\cite{Harrison_Reiman_1981} for more details on this equivalence. It belongs to a much broader class of systems, often referred to as an \textit{ADN system} (after Agmon, Douglis, and Nirenberg). These boundary value problems have received considerable attention and generated an extensive literature in classical analysis~\cite{Agmon_Douglis_Nirenberg_1959,Agmon_Douglis_Nirenberg_1964,Dauge_1988,Grisvard_2011}. This alternative formulation will prove particularly convenient for the proof of Theorem~\ref{thm:asymp}.
\end{remark}

The following corollary relates the density $\pi$ of the stationary distribution $\boldsymbol{\pi}$ to the densities $\pi_\pm$ associated with the boundary measures $\boldsymbol{\pi}_\pm$.

\begin{corollary} For almost every $u \in \mathbb{R}_\pm$, we have $\pi_\pm(u) = \frac{1}{2}\pi(u,0)$.
\end{corollary}
\begin{proof} 
    By the initial value theorem, 
    $$\lim_{y\to-\infty}y\phi(x,y)=-\int_{\mathbb R}e^{xu}\pi(u,0)\mathrm{d}u.$$
    Dividing the functional equation of Proposition~\ref{prop:FE} by $y < 0$, and letting it tend to $-\infty$, one obtains
    \begin{align*}
        0 &= -\frac{1}{2}\int_{\mathbb R}\pi(u,0)e^{xu}\mathrm{d}u+\sum_{\pm}\int_{\mathbb R_\pm}\pi_\pm(u)e^{xu}\mathrm{d}u\\
          &= \sum_{\pm}\left(\int_{\mathbb R_\pm}\left(\pi_\pm(u)-\frac{\pi(u,0)}{2}\right)e^{xu}\mathrm{d}u\right).
    \end{align*}
    Equivalently, the two-sided Laplace transform of the function 
    $$F(u):=\begin{cases}
         \pi_-(u)-\frac{1}{2}\pi(u,0) & \text{ if } u<0,  \\[2mm]
         \pi_+(u)-\frac{1}{2}\pi(u,0) & \text{ if } u>0,
    \end{cases}$$
    is the zero function. By the injectivity of this transformation, $F=0$ almost everywhere. 
\end{proof}

The formulation of the Boundary Value Problem becomes more convenient if we set $\varphi_\pm(z):=\phi_\pm(iz)$ for all $z$ for which the right-hand side is well-defined.

\begin{proposition}[Discontinuous BVP]\label{prop:bvp_disc} The functions $\varphi_+$ and $\varphi_-$ are respectively analytic in the upper and the lower half-plane, and satisfy the boundary condition
    \begin{equation}\label{eq:BVP}
    \varphi_+(t)=G(t)\varphi_-(t),
    \end{equation}
    for all $t\in\mathbb R$, where 
    \begin{equation}\label{eq:G}
        G(t):=-\frac{r_- it-\mu_2- \sqrt{\mu_2^2-2\mu_1it+t^2}}{r_+it-\mu_2- \sqrt{\mu_2^2-2\mu_1it+t^2}}.
    \end{equation}
\end{proposition}

\begin{proof} For any $x\in \mathbb C$, the equation $K(x,\,\cdot\,)=0$ admits two roots (counted with multiplicity), given by 
$$Y^\pm(x)=-\mu_2\pm \sqrt{\mu_2^2-2\mu_1x-x^2}.$$
A direct computation shows that
$$\mathfrak{Re}\left(Y^-(it)\right) = -\mu_2 - \sqrt{\frac{\mu_2^2 + t^2 + \sqrt{(\mu_2^2 + t^2)^2 + 4\mu_1^2 t^2}}{2}}\leqslant 0,$$
so that we can evaluate~\eqref{eq:FE} at $(x,y)=(it,Y^-(it))$ to obtain
$$0=\sum_\pm k_\pm(it,Y^-(it))\varphi_\pm(t).$$
Isolating $\varphi_+$ leads to \eqref{eq:BVP}, with
$$G(t):=-\frac{k_-(it,Y^-(it))}{k_+(it,Y^-(it))}=-\frac{r_-it-\mu_2-\sqrt{\mu_2^2-2\mu_1 it +t^2}}{r_+it-\mu_2-\sqrt{\mu_2^2-2\mu_1 it +t^2}}.$$
A priori, this last step is only valid for $t\ne 0$ since the denominator vanishes at the origin. However, Lemma~\ref{lemma:G(R)} establishes that the boundary condition extends continuously to $t=0$, alongside other properties of $G$. The domain of analyticity of $\varphi_\pm$ is obtained by rotating the one of $\phi_\pm$.
\end{proof}
\begin{lemma}[About the curve $G(\mathbb R)$] \label{lemma:G(R)}The coefficient $G$ and its image $G(\mathbb R)\subset\mathbb C$ satisfy the following set of properties:
    \begin{enumerate}
    \item $G$ is \textit{hermitian}, that is $G(-t)=\overline{G(t)}$ for all $t\in\mathbb R$, and $G(\mathbb R)$ is symmetric with respect to the real axis.
    \item $G$ can be continuously extended at $t=0$ with $G(0)=-\displaystyle\frac{\mu_1-r_-\mu_2}{\mu_1-r_+\mu_2}\in\mathbb R_{+}^*$.
        \item The endpoints of $G(\mathbb R)$ are given by $G(\pm \infty):=\lim_{t\to\pm\infty}G(t)=\displaystyle -\frac{r_-\pm i}{r_+\pm i}$.
        \item The intersection $G(\mathbb R)\cap \mathbb R_{-}$ is empty.
    \end{enumerate}
\end{lemma}
\noindent The curve $G(\mathbb R)$ is plotted in Figure~\ref{fig:G(R)} for various parameter sets. 

\begin{proof} 
\begin{enumerate}
    \item One can check that $Y^-(-it)=\overline{Y^-(it)}$, which immediately implies, by the definition of $G$, that $G(-t)=\overline{G(t)}$. As a direct consequence, the image $G(\mathbb R)$ is symmetric with respect to the real axis.
    \item Near $t=0$, a direct computation using the first-order Taylor expansion $\sqrt{\mu_2^2-2\mu_1it+t^2} = -\mu_2+i\frac{\mu_1}{\mu_2}t+o(t)$ shows that
    $$G(t)=-\frac{\displaystyle\left(r_--\frac{\mu_1}{\mu_2}\right) +o(1)}{\displaystyle\left(r_+-\frac{\mu_1}{\mu_2}\right)+o(1)}\xrightarrow[t\to 0]{}-\frac{\mu_1-r_-\mu_2}{\mu_1-r_+\mu_2}.$$
    The fact that $G(0)>0$ follows directly from the recurrence relation~\eqref{eq:rec}.
    \item A similar computation yields the value of $G(+\infty)$, from which we deduce that of $G(-\infty)=\overline{G(+\infty)}$.
    \item It will be useful to decompose $G(t)$ into its real and imaginary parts. First, let $\Delta=\Delta(t):=\mu_2^2-2\mu_1it+t^2$ (so that $Y^-(it)=-\mu_2-\sqrt{\Delta}$) and write $\sqrt{\Delta}=u+iv$. Then,
$$G(t)=-\frac{(\mu_2+u)^2+(r_-t-v)(r_+t-v)}{(\mu_2+u)^2+(r_+t-v)^2}+i\frac{(\mu_2+u)(r_--r_+)t}{(\mu_2+u)^2+(r_+t-v)^2}.$$
Here, $u=u(t)$ and $v=v(t)$ are explicitly given by
$$\left\{\begin{array}{l}
\displaystyle u(t)=\sqrt{\frac{\mu_2^2+t^2+\sqrt{(\mu_2^2+t^2)^2+4\mu_1^2t^2}}{2}},   \\
     \displaystyle v(t)=-\mathrm{sgn}(\mu_1t)\sqrt{\frac{-\mu_2^2-t^2+\sqrt{(\mu_2^2+t^2)^2+4\mu_1^2t^2}}{2}}.
\end{array}\right.$$
    The imaginary part of $G$ vanishes if, and only if, $t=0$, $r_+=r_-$ or $\mu_2+u=0$. Let us show that the last two alternatives never happen. On the one hand, according to the recurrence condition~\eqref{eq:rec}, $r_-\ne r_+$. On the other hand, the inequality
    $$u(t)=\sqrt{\frac{\mu_2^2+t^2+\sqrt{(\mu_2^2+t^2)^2+4\mu_1^2t^2}}{2}}\geqslant \sqrt{\frac{\mu_2^2+\sqrt{\mu_2^4}}{2}}=|\mu_2|,$$
    holds, with equality if and only if $t=0$. It follows that $\mathfrak{Im}(G(t))=0$ is equivalent to $t=0$. Furthermore, $G(0)>0$ according to the second point of this lemma, which leaves $G(\mathbb R)\cap\mathbb R_{-}$ empty. \qedhere
\end{enumerate}
\end{proof}

\begin{remark}
    Note that, due to the recurrence condition~\eqref{eq:rec}, $r_-\ne r_+$ and hence $G(-\infty)\ne G(+\infty)$. This discontinuity of the coefficient $G$ at infinity justifies the terminology \textit{discontinuous} BVP. 
\end{remark}

\begin{figure}[t]
    \centering
    \includegraphics[width=\linewidth]{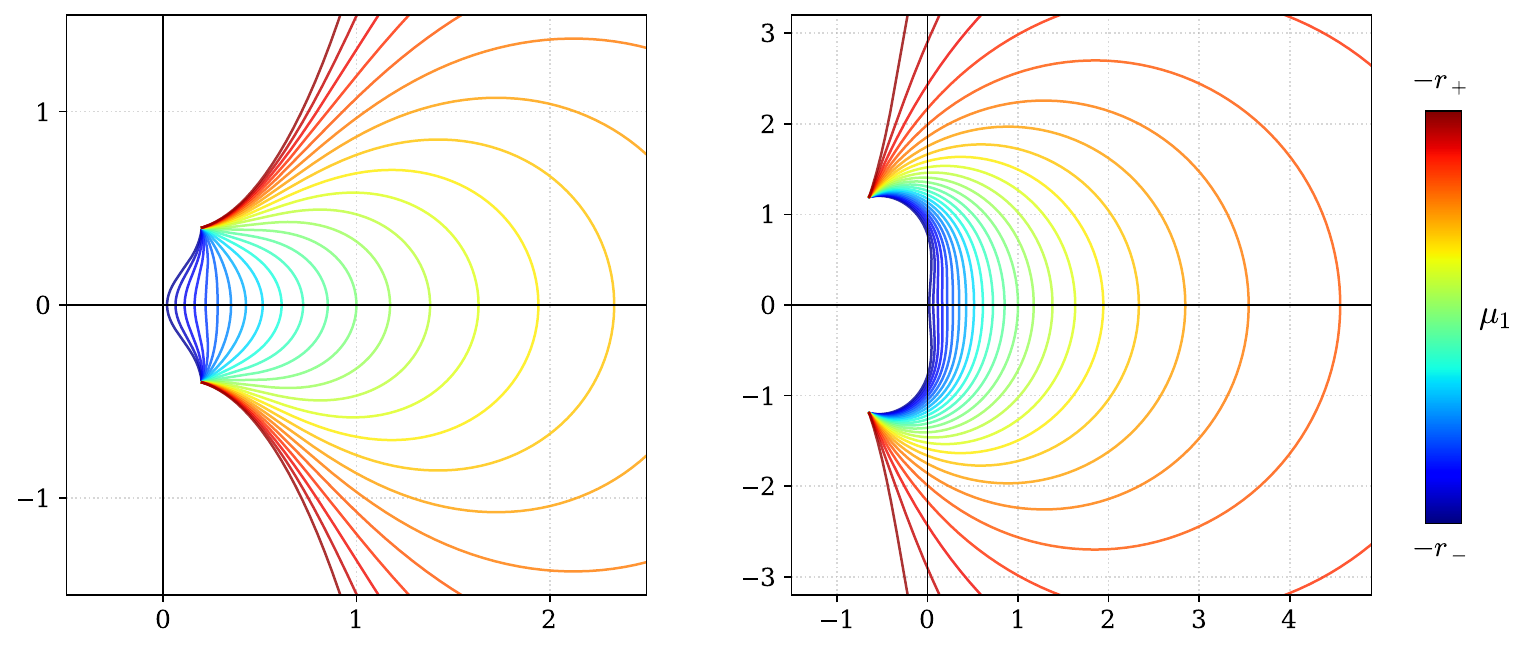}
    \caption{The curves $G(\mathbb{R})$ for $\mu_2=-1$ and $r_-=1$. The plots show the cases $r_+ = -3$ (\textit{left}) and $r_+ = -0.3$ (\textit{right}), where $\mu_1$ varies between the bounds $-r_-$ (\textit{blue}) and $-r_+$ (\textit{red}) imposed by the recurrence condition~\eqref{eq:rec}.}
    \label{fig:G(R)}
\end{figure}

Many theoretical results are stated for cases where the endpoints of $G(\mathbb R)$ meet. A first step is therefore to \textit{correct} the variation of the argument between $-\infty$ and $+\infty$. More precisely, let us introduce, for any continuous function $f:\mathbb R\to \mathbb C^*$ its  \textit{index}
$$\mathrm{Ind}(f)=\frac{1}{2i\pi}\big[\log f\big]_{-\infty}^{+\infty}=\frac{1}{2\pi}\big[\mathrm{arg}\; f\big]_{-\infty}^{+\infty}\in \mathbb R\cup\{\pm\infty\},$$
which corresponds to the variation of the argument of $f(t)$ as $t$ runs from $-\infty$
 to $+\infty$, divided by $2\pi$. One major property of the index is the following: if $f$ and $g$ are both continuous functions from $\mathbb R$ to $\mathbb C^*$, then the index of their product is the sum of the indices:
 $$\mathrm{Ind}(fg)=\mathrm{Ind}(f)+\mathrm{Ind}(g).$$
The following proposition provides the value of the index of $G$ and, in doing so, provides an insightful interpretation of Varadhan and Williams' parameter $\alpha$, as defined in \eqref{eq:alpha}.
\begin{proposition}[Index of $G$]
    $\mathrm{Ind}(G)=1-\alpha$.
\end{proposition}
Note that since $\alpha\in(0,1)$, we have $\mathrm{Ind}(G)\in(0,1)$, which is therefore never an integer.
\begin{proof} According to Lemma~\ref{lemma:G(R)}, $G(t)\notin \mathbb R_{-}$ for all $t\in\mathbb R$, hence there exists a continuous determination of its argument. Consequently, the variation of argument reduces to the difference of the arguments at its endpoints, the values of which are given by the same Lemma:
\begin{equation*}
    2\pi\mathrm{Ind}(G)=\big[\mathrm{arg}\;G\big]_{-\infty}^{+\infty}=\mathrm{arg}\big(G(+\infty)\big)-\mathrm{arg}\big(G(-\infty)\big)=2\mathrm{arg}\left(-\frac{r_-+i}{r_++i}\right).
\end{equation*}
Now, basic trigonometry leads to
\begin{align*}
\mathrm{arg}\left(-\frac{r_-+i}{r_++i}\right)&=\pi+\mathrm{arg}(r_-+i)-\mathrm{arg}(r_++i)\\
&=\pi+\left(\frac{\pi}{2}-\arctan(r_-)\right)-\left(\frac{\pi}{2}-\arctan(r_+)\right)\\
&=\pi+\arctan(r_+)-\arctan(r_-).
\end{align*}
Comparing the resulting expression of $\mathrm{Ind}(G)$ with the value of $\alpha$ given by~\eqref{eq:alphabis} yields the desired result.
\end{proof}

\begin{figure}[t]
    \centering
    \includegraphics[width=\linewidth]{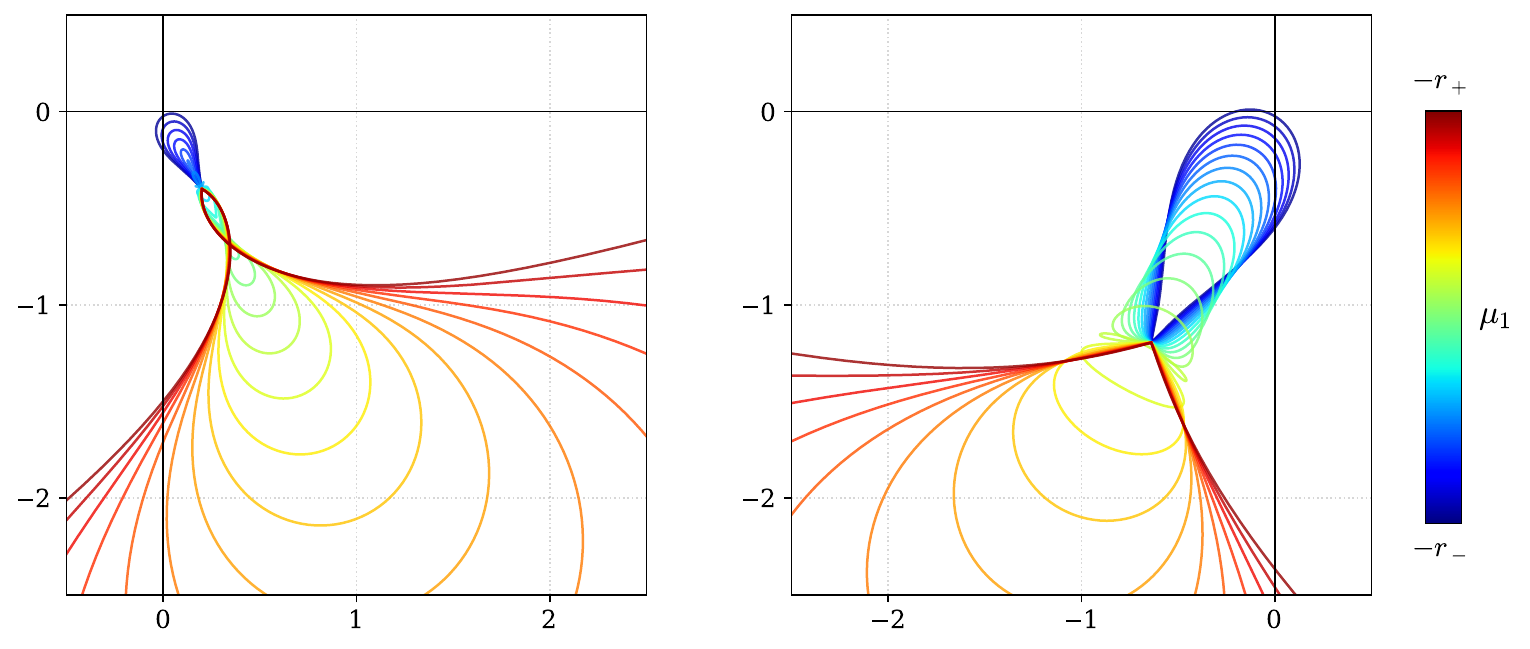}
    \caption{The corrected curves $\widetilde{G}(\mathbb{R})$ defined by~\eqref{eq:G_tilde}, using the same parameters and color coding as in Figure~\ref{fig:G(R)}.}
    \label{fig:Galpha(R)}
\end{figure}

Let $\varepsilon_\alpha$ be the \textit{correction factor}, given by the unique continuous determination of the function 
$$t \mapsto \left(\frac{t-i}{t+i}\right)^{\alpha-1}$$ 
on $\mathbb{R}$ that satisfies $\lim_{t \to -\infty} \varepsilon_\alpha(t) = 1$. Note that this choice is arbitrary, as one could equivalently replace $i$ with $z_0$ in the numerator, and $-i$ with $\overline{z}_0$ in the denominator, for any $z_0$ in the upper half-plane. The sole purpose here is to obtain a function with a constant modulus, whose total argument variation as $t$ ranges over $\mathbb{R}$ is exactly $2\pi(\alpha-1)$.

\begin{remark}[Trigonometric representation] To be completely explicit, one can express this determination using standard trigonometric functions. For $t \in \mathbb R$, the ratio $(t-i)/(t+i)$ traces the unit circle in the complex plane. The unique continuous argument $\theta(t)$ of this ratio that satisfies the limit condition is exactly given by $\theta(t) = 2\arctan(t) + \pi$. Therefore, the function $\varepsilon_\alpha$ can be equivalently defined by the closed-form formula:
\begin{equation}\label{eq:eps_alpha_explicit}
        \varepsilon_\alpha(t) = \exp\Big(i(\alpha-1)\big(2\arctan(t) + \pi\big)\Big).
    \end{equation}
    In particular, $\lim_{t\to+\infty}\varepsilon_\alpha(t)=e^{
2i\pi(\alpha-1)}$.
\end{remark}

We also define the \textit{corrected} version of $G$ to be
\begin{equation}\label{eq:G_tilde}
    \widetilde{G}(t)=\varepsilon_\alpha(t)G(t).
\end{equation}

\begin{proposition}[$\log\widetilde{G}$]\label{prop:logGtilde} There exists a continuous determination $\log \widetilde{G}$ of the logarithm of $\widetilde{G}$, and $\log \widetilde{G}(-\infty)=\log\widetilde{G}(+\infty)$.
\end{proposition}
\begin{proof}
    By construction, $\mathrm{Ind}(\varepsilon_\alpha) = (1-\alpha)\cdot\mathrm{Ind}(\varepsilon_{0}) = -(1-\alpha)$. Therefore,
    $$\mathrm{Ind}(\widetilde{G}) = \mathrm{Ind}(\varepsilon_\alpha) + \mathrm{Ind}(G) = -(1-\alpha) + (1-\alpha) = 0,$$
    where we used the fact that $\mathrm{Ind}(fg) = \mathrm{Ind}(f) + \mathrm{Ind}(g)$. This vanishing index guarantees the existence of a continuous determination of $\log \widetilde{G}$. Moreover, the zero index implies that $\widetilde{G}(-\infty)$ and $\widetilde{G}(+\infty)$ share the same argument. Since $|\varepsilon_\alpha| \equiv 1$, they also share the same modulus. We conclude that $\widetilde{G}(-\infty) = \widetilde{G}(+\infty)$, and consequently, $\log \widetilde{G}(-\infty) = \log \widetilde{G}(+\infty)$.
\end{proof}

Now, we can formulate a new Riemann boundary value problem, which will be continuous on $\overline{\mathbb R}$. Set 
$$\psi_\pm(z):=(z\pm i)^{1-\alpha}\varphi_\pm(z),$$
once again for all $z\in\mathbb C$ such that the right-hand side is well-defined.

\begin{proposition}[Continuous BVP]\label{prop:bvp_cont} The functions $\psi_+$ and $\psi_-$ are respectively analytic in the upper and the lower half-plane, and satisfy the boundary condition
\begin{equation}\label{eq:bvp_cont}
    \psi_+(t)=\widetilde{G}(t)\psi_-(t),
\end{equation}
where $\widetilde{G}$ is given by~\eqref{eq:G_tilde}.
\end{proposition}

\begin{proof}
    Recall from Proposition~\ref{prop:bvp_disc} that $\varphi_+$ is analytic in the upper half-plane. The principal branch of $t \mapsto (t+i)^{1-\alpha}$ is also analytic in the upper half-plane, as its branch point $t=-i$ lies strictly below the real axis. Therefore, their product $\psi_+$ is analytic in the upper half-plane. The analyticity of $\psi_-$ in the lower half-plane follows by a symmetric argument. For $t\in\mathbb R$, the boundary condition is a straightforward rewriting of the original boundary condition $\varphi_+(t)=G(t)\varphi_-(t)$ of Proposition~\ref{prop:bvp_disc}:
    \begin{equation*}
        \psi_+(t) = \varphi_+(t)(t+i)^{1-\alpha} = G(t)\varphi_-(t)(t+i)^{1-\alpha} = G(t)\left(\frac{t-i}{t+i}\right)^{\alpha-1}\psi_-(t).
    \end{equation*}
    Recalling the definitions of $\varepsilon_\alpha(t)$ and $\widetilde{G}(t)$, this immediately simplifies to~\eqref{eq:bvp_cont}, completing the proof.
\end{proof}

Finally, the following lemma provides a strictly sublinear growth bound for $\psi_\pm$ within its domain of convergence, a result that will be needed in the proof of Theorem~\ref{thm:laplace}.

\begin{lemma}[Growth of $\psi_\pm$]\label{lemma:growth} As $|z|\to+\infty$ within their domains of convergence, $\psi_\pm(z)=o(|z|)$.
\end{lemma}

\begin{proof} Being Laplace transforms of finite measures, the functions $\phi_\pm$ are bounded by their value at $0$. Indeed, for any valid $x$, we have:
$$|\phi_\pm(x)|=\left|\int_{\mathbb R_\pm}e^{xu}\mathrm{d}\boldsymbol{\pi}_{\pm}(u)\right|\leqslant \int_{\mathbb R_\pm}e^{\mathfrak{Re}(x)u}\mathrm{d}\boldsymbol{\pi}_\pm(u)\leqslant \boldsymbol{\pi}_\pm \left(\mathbb R_\pm\right).$$
This total mass is finite and can be explicitly computed (see the proof of Theorem~\ref{thm:laplace}). Recalling the definition of $\psi_\pm$ we deduce the following asymptotic behavior:
$$\psi_{\pm}(z)=(z\pm i)^{1-\alpha}\phi_\pm(iz)=O(|z|^{1-\alpha})=o(|z|),$$
where the last equality follows from the recurrence condition~\eqref{eq:rec}, which enforces $\alpha\in (0,1)$.
\end{proof}

\section{Sokhotski-Plemelj formula}\label{sec:sp}

This section provides the solution to the continuous boundary value problem of Proposition~\ref{prop:bvp_cont}, which in turn yields the solution to the discontinuous stated at Proposition~\ref{prop:bvp_disc}. The strategy is to construct sectionally analytic functions on $\mathbb C\setminus\mathbb R$ using Cauchy-type integrals, whose limits from above and below satisfy the boundary condition of the BVP. We introduce the following regularity class, which ensures that the Cauchy-type integrals are well-defined and that the Sokhotski-Plemelj formulas hold.

\begin{definition}[Hölder condition $\breve{H}^\mu$]\label{def:holder} A function $f:\mathbb R\to\mathbb C$ is said to satisfy the Hölder condition $\breve{H}^\mu$ with $0<\mu\leqslant 1$ if there exists a closed interval $I\subset \mathbb R$ containing $0$ (as interior point) such that
\begin{itemize}
    \item[(i)] $f$ is $\mu$-Hölder continuous on $I$, \textit{i.e.} there exists $\kappa>0$ such that
    $$|f(s)-f(t)|\leqslant \kappa|s-t|^\mu,\quad \forall s,t\in I,$$
    \item[(ii)] $f$ is $\mu$-Hölder continuous in the neighborhood of infinity, \textit{i.e.} there exists $\kappa'>0$ such that
    $$|f(s)-f(t)|\leqslant \kappa'\left|\frac{1}{s}-\frac{1}{t}\right|^\mu,\quad \forall s,t\in\mathbb R\setminus I .$$
\end{itemize}
\end{definition}

This technical condition amounts to transforming the function $f:\mathbb R\to\mathbb C$ into a function $f^\circ$ defined on a circle (via a Möbius transformation) and requiring $f^\circ$ to be $\mu$-Hölder continuous. See~\cite[Lemma 5.1.2]{lu94} for more details on this equivalence. The following proposition is a classical result in complex analysis; we refer the reader to~\cite[\S 14.7]{gakhov}, which addresses this specific setting.

\begin{proposition}[Sokhotski-Plemelj formula]\label{prop:SP} Let $f\in \breve{H}^\mu$ for some $\mu\in (0,1]$. The function
$$F(z):=\frac{1}{2i\pi}\int_{\mathbb R}\frac{ f(\tau)}{\tau-z}\mathrm{d}\tau,\quad z\in\mathbb C\setminus \mathbb R,$$
is sectionally analytic, and admits, for all $t\in\mathbb R$, a limit from below $F_-(t)$ and a limit from above $F_+(t)$, that satisfy
\begin{equation}\label{eq:jump_cond}
    F_+(t)-F_-(t)=f(t),\quad F_+(t)+F_-(t)=\frac{1}{i\pi }\;\mathrm{p.v.}\int_{\mathbb R}\frac{f(\tau)}{\tau-t}\mathrm{d}\tau.
\end{equation}
Moreover, $F$ is bounded at infinity.
\end{proposition}
Note that Equation~\eqref{eq:jump_cond} involves the Cauchy principal value defined, in this context, by
$$\mathrm{p.v.}\int_{\mathbb R}\frac{f(\tau)}{\tau-t}\mathrm{d}\tau :=\lim_{\substack{\varepsilon\to 0^+ \\[0.08cm] R\to+\infty}}\left\{\int_{t-R}^{t-\varepsilon}\frac{f(\tau)}{\tau-t}\mathrm{d}\tau+\int_{t+\varepsilon}^{t+R}\frac{f(\tau)}{\tau-t}\mathrm{d}\tau\right\}.$$

\begin{remark}[from multiplicative to additive]
The Sokhotski-Plemelj formula provides the solution to \textit{additive} boundary value problems. To recast our multiplicative problem into this framework, we take the logarithm of the continuous boundary value problem~\eqref{eq:bvp_cont}. This operation is well-defined thanks to Proposition~\ref{prop:logGtilde}, and yields
\begin{equation}\label{eq:additive_bvp}
    \log\psi_+(t)-\log\psi_-(t)=\log\widetilde{G}(t).
\end{equation}
The following lemma verifies that we can indeed apply the Sokhotski-Plemelj formula to $f(t):=\log\widetilde{G}(t)$.
\end{remark}

\begin{lemma}[$\widetilde{G}\in\breve{H}^1$]\label{lemma:holder} The function $\log\widetilde{G}$ satisfies the Hölder condition $\breve{H}^1$.
\end{lemma}

\begin{proof} Note that $\widetilde{G}=\varepsilon_\alpha G$ is Lipschitz on $\mathbb R$, since both factors are bounded and have bounded derivatives; therefore, it is $1$-Hölder continuous. Let us show that it is also Hölder continuous at infinity. As $t\to  \pm\infty$,
$$\widetilde{G}(t)=-\frac{r_--i}{r_+-i}\left(1-\frac{C_\pm}{t}\right)+O\left(t^{-2}\right),$$
where 
$$C_\pm:=i\left(2(\alpha-1)+\frac{(r_+-r_-)(\mu_2\mp i\mu_1)}{(ir_-\mp 1)(ir_+\mp 1)}\right).$$
These asymptotic expansions directly imply the second condition of Definition~\ref{def:holder}. Finally, according to Lemma~\ref{lemma:G(R)}, $G$ is bounded away from $0$, and so is $\widetilde{G}$ (because $|\varepsilon_{\alpha}|=1$). Because the logarithm is Lipschitz continuous on any domain bounded away from zero, $\log \widetilde{G}$ also satisfies $\breve{H}^\mu$.
\end{proof}

The fact that the derivative of $\widetilde{G}$ is bounded is illustrated in Figure~\ref{fig:G'(R)}.

\begin{figure}[t]
    \centering
    \includegraphics[width=\linewidth]{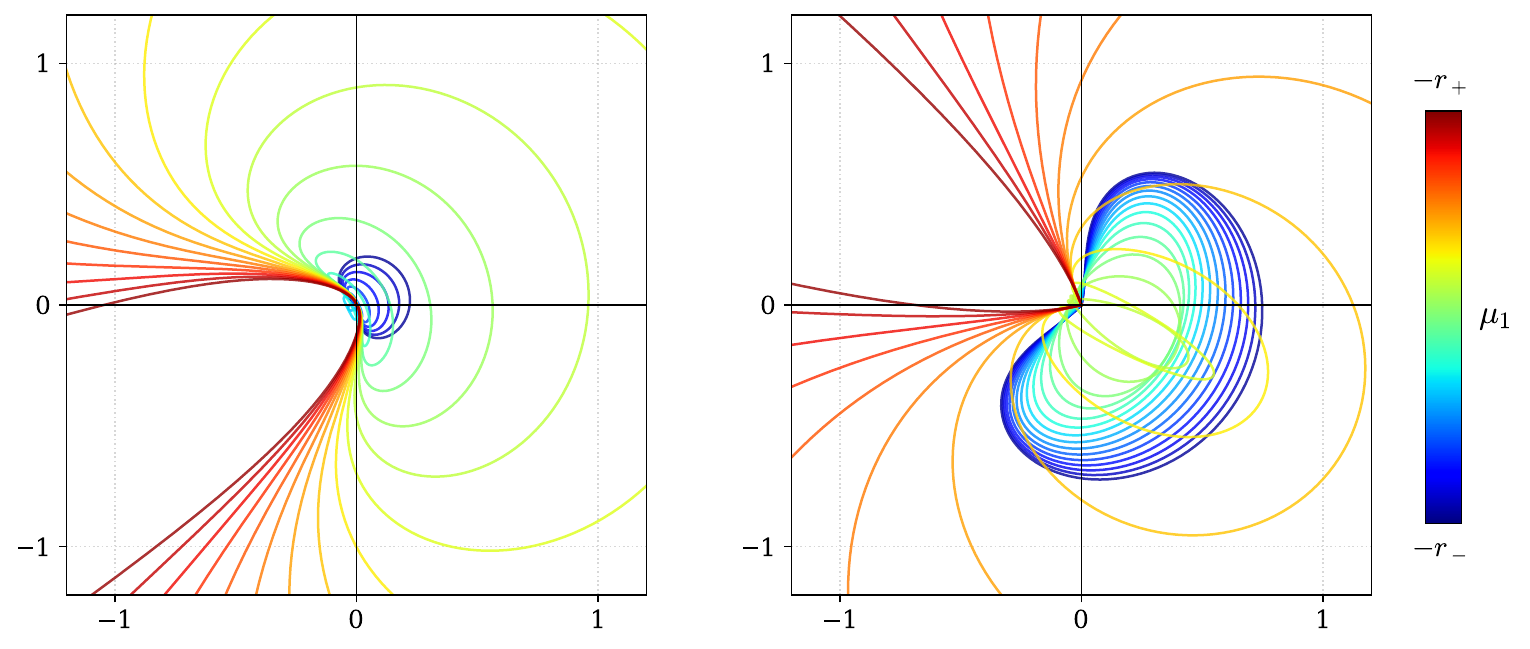}
    \caption{The curves $\widetilde{G}'(\mathbb R)$, using the same parameters and color coding as in Figures~\ref{fig:G(R)} and~\ref{fig:Galpha(R)}. Boundedness of the derivative ensures that $\widetilde{G}$ is Lipschitz continuous, \textit{i.e.}, $1$-Hölder continuous on $\mathbb R$.}
    \label{fig:G'(R)}
\end{figure}

\begin{proposition}[Solutions for the continuous BVP]\label{prop:bvp_sol} Let $\Psi$ be a solution to the continuous BVP of Proposition~\ref{prop:bvp_cont} satisfying $\Psi(z)=o(|z|)$ at infinity. Then, for $z\in\mathbb C\setminus \mathbb R$,
$$\Psi(z)\propto \exp\left(\frac{1}{2i\pi}\int_\mathbb R \frac{\log\widetilde{G}(\tau)}{\tau-z}\mathrm{d}\tau\right).$$
\end{proposition}
\begin{proof}
   According to Lemma~\ref{lemma:holder}, one can apply the Sokhotski-Plemelj formula (Proposition~\ref{prop:SP}) to define the following sectionally analytic function 
   $$F(z):=\frac{1}{2i\pi}\int_{\mathbb R}\frac{\log\widetilde{G}(\tau)}{\tau-z}\mathrm{d}\tau,$$
   whose limits from above and below satisfy, for any $t\in\mathbb R$, 
   $$F_+(t)-F_-(t)=\log\widetilde{G}(t).$$
   This implies that $\exp(F)$ is a non-vanishing solution to the continuous BVP of Proposition~\ref{prop:bvp_cont}. Consider the ratio $R:=\Psi/\exp(F)$. The function $R$ is analytic in $\mathbb C\setminus \mathbb R$, and for every $t\in\mathbb R$,
   $$R_+(t)=\frac{\Psi_+(t)}{\exp F_+(t)}=\frac{\widetilde{G}(t)\Psi_-(t)}{\widetilde{G}(t)\exp F_-(t)}=\frac{\Psi_-(t)}{\exp F_-(t)}=R_-(t).$$
    Thus, $R$ is continuous across the real axis. By Morera's theorem, $R$ extends to an entire function on $\mathbb C$. We know from Proposition~\ref{prop:SP} that $\exp F$ is bounded at infinity, and by hypothesis that $\Psi(z)=o(|z|)$, hence $R(z)=o(|z|)$. By Liouville's theorem, $R$ must be constant, which yields the stated result.
\end{proof}

\begin{theorem}[Explicit Laplace transforms]\label{thm:laplace} 
The Laplace transforms $\phi_\pm(x)$ are given by the following explicit formulas:
\begin{enumerate}
    \item[(a)] For $\pm \mathfrak{Re}(x) < 0$,
    \begin{equation*}
        \phi_\pm(x) = \frac{\Lambda}{(x \mp 1)^{1-\alpha}} \exp\left(\frac{1}{2i\pi} \int_{\mathbb{R}} \frac{\log\widetilde{G}(\tau)}{\tau+ix} \mathrm{d}\tau \right).
    \end{equation*}
    \item[(b)] On the imaginary axis, for $x=it$ with $t \in \mathbb{R}$,
    \begin{equation*}
        \phi_\pm(it) = \frac{\Lambda}{(it \mp 1)^{1-\alpha}} \widetilde{G}(t)^{\pm 1/2} \exp\left(\frac{1}{2i\pi} \,\mathrm{p.v.} \int_{\mathbb{R}} \frac{\log\widetilde{G}(\tau)}{\tau-t} \mathrm{d}\tau \right),
    \end{equation*}
    where the principal value integral is defined as in Proposition~\ref{prop:SP}.
\end{enumerate}
Here, the common normalization constant $\Lambda$ is uniquely determined by the relation
\begin{equation}\label{eq:normalization}
    \phi_\pm(0) = \pm \frac{\mu_1 - \mu_2 r_\mp}{r_- - r_+}.
\end{equation}
\end{theorem}
\begin{proof} (a) The function $\Psi:\mathbb C\setminus \mathbb R\to\mathbb C$ that coincides with $\psi_+$ on $\{\mathfrak{Im}(z)>0\}$ and with $\psi_-$ on $\{\mathfrak{Im}(z)<0\}$ is a solution to the continuous BVP of Proposition~\ref{prop:bvp_cont}, and satisfies $\Psi(z)=o(|z|)$ as $|z|\to+\infty$ (see Lemma~\ref{lemma:growth}). We can therefore apply Proposition~\ref{prop:bvp_sol}:
    $$\psi_\pm(z)\propto \exp\left(\frac{1}{2i\pi}\int_\mathbb R \frac{\log \widetilde{G}(\tau)}{\tau-z}\mathrm{d}\tau\right).$$
    Thus, for any $x$ in the respective domains of convergence,
    \begin{align*}
        \phi_\pm(x) = \frac{\psi_\pm(-ix)}{(-ix \pm i)^{1-\alpha}}\propto \frac{1}{(x\mp 1)^{1-\alpha}}\exp\left(\frac{1}{2i\pi}\int_{\mathbb R}\frac{\log\widetilde{G}(\tau)}{\tau+ix}\mathrm{d}\tau\right).
    \end{align*}
    (b) Applying the Sokhotski-Plemelj formula (Proposition~\ref{prop:SP}) to the Cauchy integral of $f(\tau)=\log\widetilde{G}(\tau)$, the limit yields the principal value integral and a boundary term $\pm\frac{1}{2}\log\widetilde{G}(t)$. Exponentiating this sum produces the factor $\widetilde{G}(t)^{\pm 1/2}$, which leads to the stated formula.\\
    To determine the normalization constant $\Lambda$, one can evaluate the functional equation~\eqref{eq:FE} at $x=0$ and divide by $y\ne 0$ to obtain
    $$0=\left(\frac{y}{2}+\mu_2\right)\phi(0,y)+\phi_-(0)+\phi_+(0).$$
    Letting $y\to 0$ leads to the first linear equation $\phi_-(0)+\phi_+(0)=-\mu_2$ (since $\phi(0,0)=\boldsymbol{\pi}(\mathbb H_+)=1$). Similarly, evaluating~\eqref{eq:FE} at $y=0$, dividing by $x\ne 0$ and letting $x\to 0$ gives the second relation $r_-\phi_-(0)+r_+\phi_+(0)=-\mu_1$. The resulting linear system 
    $$\left\{\begin{array}{ccccc}
         r_-\phi_-(0)&+&r_+\phi_+(0)&=&-\mu_1  \\
         \phi_-(0)&+&\phi_+(0)&=&-\mu_2
    \end{array}\right.$$
    has determinant $R_-^\intercal R_+=r_--r_+$, which is non-zero according to the recurrence relation~\eqref{eq:rec}; it therefore admits the stated solution as its unique solution. Note that the same recurrence condition ensures that $\phi_\pm(0)>0$.
\end{proof}

In certain special cases, the integral expression of Theorem~\ref{thm:laplace} reduces to a simple closed-form formula, which can be explicitly inverted to recover the density of the invariant measure. The following proposition details one such case.

\begin{proposition}[A symmetric case]\label{prop:ex} If $(\mu_1,\mu_2,r_-,r_+)=(0,-1,1,-1)$, then the density $\pi$ of the invariant measure is given by
$$\pi(u,v)=\frac{1}{\sqrt{\pi}}\frac{\sqrt{\sqrt{u^2+v^2}+v}}{\sqrt{u^2+v^2}}\exp\left(-\sqrt{u^2+v^2}-v\right),$$
or equivalently, in polar coordinates
$$\pi(r\cos(\theta),r\sin(\theta))=\sqrt{\frac{2}{\pi r}}\cos\left(\frac{\theta}{2}-\frac{\pi}{4}\right)\exp\left(-2r\cos^2\left(\frac{\theta}{2}-\frac{\pi}{4}\right)\right).$$
\end{proposition}

\begin{proof}
    With this choice of parameters, the function $\widetilde{G}(t)$ evaluates to a constant (in fact, it is the only choice that leads to such simplification). Hence, Theorem~\ref{thm:laplace} yields 
    \begin{equation}\label{eq:lateral_laplace}
        \phi_\pm(x)=\frac{1}{2\sqrt{1\mp x}}.
    \end{equation}
    We recognize here the Laplace transforms of a Gamma distribution, and the corresponding lateral densities are given explicitly by
    $$\pi_\pm(u)=\frac{1}{2\sqrt{\pi}}\frac{\exp(-|u|)}{\sqrt{|u|}}.$$
    Injecting~\eqref{eq:lateral_laplace} into the functional equation~\eqref{eq:FE} leads to the following expression for the bivariate Laplace transform:
    \begin{equation}\label{eq:laplace_symmetric}
    \phi(x,y)=\left(\frac{1}{\sqrt{1-x}}+\frac{1}{\sqrt{1+x}}\right)\frac{1}{1+\sqrt{1-x^2}-y}.
    \end{equation}
    First, we invert $\phi(x,y)$ with respect to the second variable, by noticing that, as a function of $y$, it is of the form $c/(y_0-y)$. Here $y_0= y_0(x)=1+\sqrt{1-x^2}$. Hence, we obtain the following \textit{partial inverse} identity
    \begin{equation*}
        \int_{-\infty}^{+\infty}\pi(u,v)e^{xu}\mathrm{d}u=\left(\frac{1}{\sqrt{1-x}}+\frac{1}{\sqrt{1+x}}\right)\exp\left(-v\left(1+\sqrt{1-x^2}\right)\right)
    \end{equation*}
    For $x=it$, this reduces our problem to the Fourier inversion of this quantity with respect to $t$:
    $$\pi(u,v)=\frac{e^{-v}}{2\pi}\int_{-\infty}^{+\infty}\left(\frac{1}{\sqrt{1-it}}+\frac{1}{\sqrt{1+it}}\right)\exp\left(-v\sqrt{1+t^2}-itu\right)\mathrm{d}t.$$
    Successively applying the substitutions $t=\sinh \xi$ and $\omega=\exp(\xi/2)$, the integral in the above expression reduces down to
    \begin{align*}
        &\quad\!\quad\int_{-\infty}^{+\infty}\left(\frac{1}{\sqrt{1-it}}+\frac{1}{\sqrt{1+it}}\right)\exp\left(-v\sqrt{1+t^2}-itu\right)\mathrm{d}t\\
        &=2\int_{-\infty}^{+\infty}\cosh\left(\frac{\xi}{2}\right)\exp\left(-v\cosh \xi -iu\sinh \xi\right)\mathrm{d}\xi\\
        &=2\int_0^{+\infty}\exp\left(-a\omega^2-\frac{\overline{a}}{\omega^2}\right)\mathrm{d}\omega+2\int_0^{+\infty}\frac{1}{\omega^2}\exp\left(-a\omega^2-\frac{\overline{a}}{\omega^2}\right)\mathrm{d}\omega\\
        &=2\int_0^{+\infty}\exp\left(-a\omega^2-\frac{\overline{a}}{\omega^2}\right)\mathrm{d}\omega+2\int_0^{+\infty}\exp\left(-\overline{a}\omega^2-\frac{a}{\omega^2}\right)\mathrm{d}\omega,
    \end{align*}
    where $a=\frac12(v+iu)$. Some standard hyperbolic trigonometric and algebraic simplifications are omitted. Here, we recognize Gaussian-type integrals, which yields
    $$\pi(u,v)=\frac{\exp(-v-2|a|)}{2\sqrt{\pi}}\left(\frac{1}{\sqrt{a}}+\frac{1}{\sqrt{\overline{a}}}\right).$$
    Minor algebraic cleanups (\textit{e.g.}, $2|a|=\sqrt{u^2+v^2}$) conclude the proof.
\end{proof}

\begin{remark} The formula from Proposition~\ref{prop:ex} can be compared with already known results:
\begin{itemize}
    \item We know from Remark~\ref{rem:1D} that $Z_t^{(2)}$ behaves like a one-dimensional Brownian motion reflected at $0$ with drift $\mu_2$. Integrating $\pi(u,v)$ with respect to $u$ gives 
\begin{align*}
    \int_{-\infty}^{+\infty}\pi(u,v)\mathrm{d}u&=\frac{2e^{-v}}{\sqrt{\pi}}\int_0^{+\infty}\frac{\sqrt{\sqrt{u^2+v^2}+v}}{\sqrt{u^2+v^2}}\exp\left(-\sqrt{u^2+v^2}\right)\mathrm{d}u\\
    &=\frac{2e^{-v}}{\sqrt{\pi}}\int_v^{+\infty}\frac{e^{-t}}{\sqrt{t-v}}\mathrm{d}t=\frac{2e^{-2v}}{\sqrt{\pi}}\Gamma(1/2)=2e^{-2v},
\end{align*}
which is precisely the density of the invariant measure of a one-dimensional Brownian motion, with drift $\mu_2=-1$, reflected at $0$ (\textit{i.e.}, the exponential distribution with parameter $-2\mu_2=2$).
\item The authors of \cite{BoMe} introduce two additional parameters $\alpha_1$ and $\alpha_2$. Applying their definitions to Proposition~\ref{prop:ex} yields $\alpha_1=\alpha_2=0$. Note that this example is discussed in Section 8.3.1 of \cite{BoMe} for a wedge opening $\beta<\pi$.
\end{itemize}
\end{remark}

\section{Tauberian theorems and asymptotic analysis}\label{sec:asymp}

From the kernel functional equation of Proposition~\ref{prop:FE} and the explicit formulas of Theorem~\ref{thm:laplace}, the behavior at the origin and the tail asymptotics of the boundary densities are derived via Tauberian theorems, as stated in the following two theorems. For compactness, we write $f\asymp g$ when both functions exhibit the same asymptotic behavior up to positive constants.

\subsection{Behaviour at the origin}

\begin{theorem}[Behavior at the origin]\label{thm:asymp} The stationary density $\pi$ exhibits a radial asymptotic behavior given by
$$\pi(r\cos\theta,r\sin \theta)\asymp \sin\left(\delta_+-\alpha\theta\right)r^{-\alpha},\quad \text{ as }r\to0^+,$$
where $\delta_+$ is the angle between the horizontal axis and the reflection vector $R_+$ (see Figure~\ref{fig:model}).
\end{theorem}

\begin{remark} It is worth noting that since $\alpha\in (0,1)$, the singularity described in Theorem~\ref{thm:asymp} is in fact integrable. Furthermore, we exactly recover the asymptotics of the invariant measure at the origin in the \textit{driftless} case, see~\cite{Williams_1985b} (note that with the author's notations, $\theta_1:=\delta_+-\frac\pi2$, which instead yields a coefficient of  $\cos(\alpha\theta-\theta_1)$). In particular, neither the exponent of the radius $r$ nor the angular function of $\theta$ depends on the drift $\mu$.
\end{remark}

\begin{proof}
    We first analyze the asymptotic behavior of the Laplace transform $\phi_+(x)$, using the explicit formula given in Theorem~\ref{thm:laplace}~(a).
According to Proposition~\ref{prop:SP}, the Cauchy-type integral is bounded as $|x| \to \infty$ within the domain of analyticity. Thus, the exponential factor converges to a finite constant $C$. The prefactor in the expression of $\phi_+(x)$ behaves as
\begin{equation*}
    \frac{\kappa}{(x-1)^{1-\alpha}} \underset{x \to -\infty}{\sim} \kappa |x|^{\alpha-1}.
\end{equation*}
Combining these two estimates,
\begin{equation*}\label{eq:phi_asymp_inf}
    \phi_+(x) \sim C\kappa|x|^{\alpha-1},\quad \text{as }|x|\to  \infty
\end{equation*}
A symmetric argument applied to $\phi_-(x)$ yields an analogous result. Let us now introduce a scaling factor $\lambda$. For fixed values of $x$ and $y$, as $\lambda \to \infty$,
$$K(\lambda x,\lambda  y)\asymp \lambda^2,\quad k_\pm(\lambda x,\lambda y)\asymp \lambda,\quad \phi_\pm(\lambda x)\asymp\lambda^{\alpha-1}.$$
Hence, from the functional equation~\eqref{eq:FE}, one obtains
$$\phi(\lambda x,\lambda y)=\frac{-1}{K(\lambda x,\lambda y)}\sum_{\pm}k_\pm(\lambda x,\lambda y)\phi_\pm(\lambda x)\asymp \lambda^{\alpha-2}.$$
Applying a two-dimensional version of the Tauberian theorem (see \cite{tauberian}, especially the section ``\textit{Many-Dimensional Tauberian Theorems}'') yields, for any fixed value of $\theta$,
$$\pi(r\cos\theta,r\sin\theta)\asymp r^{-\alpha},\quad \text{as }r\to 0^+,$$
\textit{i.e.}, there exists a function $c:[0,\pi]\to \mathbb R_+$ such that $\pi(r\cos\theta,r\sin\theta)\asymp c(\theta)r^{-\alpha}$, that we must now determine. To do so, recall that according to Remark~\ref{rem:ADN}, the density $\pi$ satisfies system~\eqref{eq:EBVP_system}. The asymptotics of solutions to such systems are well understood, see in particular~\cite{Kondratiev_1967,Costabel_Dauge_1993}, whose results we will use as a black box. The main idea is that, having identified the leading term $$\pi_0(r\cos\theta,r\sin\theta):=c(\theta)r^{-\alpha}$$ in the expansion of $\pi$ near the origin, the function $c(\theta)$ is automatically analytic (see \cite[Thm. 1.2]{Kondratiev_1967} or the introduction of~\cite{Costabel_Dauge_1993}), and $\pi_0$ satisfies a system similar to~\eqref{eq:EBVP_system}, obtained by keeping only the highest-order terms of each differential operator:
\begin{equation}\label{eq:EBVP_system_alt}
\left\{\begin{array}{ll}
\Delta\pi_0(u,v)=0 & \text{for all } (u,v)\in \overset{\circ}{\mathbb{H}}_+, \\[8pt]
\left(\displaystyle\frac{\partial}{\partial v}-r_\pm\frac{\partial}{\partial u}\right)\pi_0(u,0)=0&\text{for all }\pm u>0.
\end{array}\right.
\end{equation}
For more details, see Section 1.d, ``\textit{Expansion of the boundary value system},'' of~\cite{Costabel_Dauge_1993} (in particular, Equation (1.12) in their Conclusion 1.1). In polar coordinates, the equation $\Delta\pi_0(u,v)=0$ can be written as
$$0=\left(\frac{\partial^2}{\partial r^2}+\frac{1}{r}\frac{\partial}{\partial r}+\frac{1}{r^2}\frac{\partial^2}{\partial \theta^2}\right)\big(c(\theta)r^{-\alpha}\big)=r^{-\alpha-2}\big(\alpha^2 c(\theta)+c''(\theta)\big).$$
This reduces to first solving $\alpha^2c+c''=0$, for which a basis of solutions is given by $\cos(\alpha \theta)$ and $\sin(\alpha\theta)$. There thus exist two constants $c_1$ and $c_2$ such that $c(\theta)=c_1\cos(\alpha\theta)+c_2\sin(\alpha\theta)$. Let us determine an equation relating $c_1$ and $c_2$ using the boundary condition on $\mathbb R_+\times \{0\}$ from~\eqref{eq:EBVP_system_alt}, which becomes, in polar coordinates:
$$0=\left(\frac{1}{r}\frac{\partial}{\partial \theta}-r_+\frac{\partial}{\partial r}\right) \pi_0(r\cos\theta,r\sin\theta)\Big|_{\theta=0}=\alpha r^{-\alpha-1}\big(c_2+r_+c_1\big),$$
yielding $c_2=-r_+c_1$, and finally
$$c(\theta)\propto \cos(\alpha\theta)-r_+\sin(\alpha\theta)\propto \sin(\alpha\theta-\delta_+),$$
since $\tan(\delta_+)=1/r_+$.
\end{proof}

\begin{example} When $(\mu_1,\mu_2,r_-,r_+)=(0,-1,1,-1)$, $\alpha=\frac{1}{2}$ and the asymptotic behavior of $\pi$ near the origin follows directly from its trigonometric form given in Proposition~\ref{prop:ex}:
$$\pi(r\cos \theta,r\sin\theta)\sim \cos\left(\frac{\theta}{2}-\frac{\pi}{4}\right)\sqrt{\frac{2}{\pi r}}=\sin\left(\frac{3\pi}{4}-\frac{\theta}{2}\right)\sqrt{\frac{2}{\pi r}},\text{ as }r\to 0^+.$$
\end{example}

\subsection{Tails asymptotics}

The tail behavior of the lateral densities $\pi_\pm(u)$ exhibits a \textit{phase transition}. To fully characterize the different asymptotic regimes, we must introduce the critical reflection parameters $r_\pm^\star$. These thresholds are entirely determined by the drift direction $\rho:=\mu_1/\mu_2$ and are given by
$$r_+^\star=-\frac{1}{r_-^\star}=\rho-\sqrt{\rho^2+1}.$$

\begin{theorem}[Exponential decay]\label{thm:infty} The lateral densities satisfy
\begin{equation}\label{eq:asymptotics}
    \pi_+(u)\asymp u^{-\kappa}\exp(-\gamma u),\quad\text{as }u\to+\infty,
\end{equation}
where the exponents $\gamma$ and $\kappa$ are given by
\begin{center}
\renewcommand{\arraystretch}{1.3}
\begin{tabular}{|c|c|c|c|}
\hline
Regime & $r_+ < r_+^\star$ & $r_+ = r_+^\star$ & $r_+ > r_+^\star$ \\
\hline 
$\gamma$ \rule[-2.5ex]{0pt}{6.5ex} & $\displaystyle\sqrt{\mu_1^2+\mu_2^2}-\mu_1$& $\displaystyle\frac{\mu_2}{r_+}$ & $\displaystyle 2\frac{r_+\mu_2-\mu_1}{1+r_+^2}$\\
\hline
$\kappa$ & $3/2$ & $1/2$ & $0$ \\
\hline
\end{tabular}
\end{center}
A symmetric result holds for $\pi_-(u)$ as $u\to -\infty$, with a threshold at $r_-=r_-^\star$.
\end{theorem}

\begin{proof}
To understand the tail behavior of $\pi_+(u)$ as $u \to+\infty$, we must study the singularities of its Laplace transform $\phi_+(x)$ strictly beyond its initial domain of convergence ($\mathfrak{Re}(x) \le 0$). According to classical asymptotic transfer theorems, the singularity in the right half-plane closest to the imaginary axis dictates the exponential decay of the original density. Through the functional equation~\eqref{eq:FE} and Morera's theorem, $\phi_+(x)$ can be meromorphically continued to the right half-plane via
$$\phi_+(x)=-\frac{k_-(x,Y^-(x))}{k_+(x,Y^-(x))}\phi_-(x).$$
This continuation faces two potential obstacles (singularities) in the right half-plane. The first is a square-root branch point 
$$x_+ := \sqrt{\mu_1^2 + \mu_2^2} - \mu_1 > 0,$$ 
that is the positive root of the discriminant of $K(x,\,\cdot \,)=0$. The second is a potential simple pole 
$$p := 2 \frac{r_+ \mu_2 - \mu_1}{1 + r_+^2},$$ 
arising from the denominator of the continuation formula.  Note that the recurrence condition~\eqref{eq:rec} ensures that $p>0$. The leftmost singularity $\gamma$ in the right half-plane is $\gamma=p$ when $r_+>r_+^\star$, and $\gamma=x_+$ when $r_+<r_+^\star$. These two singularities coincide at the critical value $r_+=r_+^\star$. In each case, the local behavior of $\phi_+(x)$ as $x\to \gamma^-$ is given by
\begin{align*}
    \phi_+(x)&\sim \frac{(r_--r_+)x_+\sqrt{2\|\mu\|_2}\phi_-(x_+)}{(r_+x_+-\mu_2)^2}\cdot\sqrt{x_+-x},& \text{if }r_+<r_+^\star,\\
    \phi_+(x)&\sim 2\frac{(r_--r_+)(r_+p-\mu_2)\phi_-(p)}{1+r_+^2}\cdot \frac{1}{p-x},&\text{if }r_+>r_+^\star,\\
    \phi_+(x)&\sim \frac{(r_--r_+)x_+\phi_-(x_+)}{\sqrt{2\|\mu\|_2}}\cdot\frac{1}{\sqrt{x_+-x}},&\text{if }r_+=r_+^\star,
\end{align*}
\textit{i.e.} $\phi_+(x)\asymp \phi_-(\gamma)(\gamma-x)^{\kappa-1}$, where the values of $\gamma$ and $\kappa$ are exactly those given by the table above. In each case, $\gamma>0$, hence cannot be a pole of $\phi_-$. Furthermore, $\phi_-(\gamma)>0$ (being the Laplace transform of a non-trivial positive measure evaluated at $\gamma>0$). By the transfer theorem for Laplace transforms (see \cite[Thm.~37.1]{doetsch}), this directly yields~\eqref{eq:asymptotics}. A symmetric argument applied to the continuation of $\phi_-(x)$ yields the analogous tail behavior for $\pi_-(u)$, concluding the proof.
\end{proof}

\begin{remark}[Geometric interpretation] The critical reflection vectors $R_\pm^\star:=(1,r_\pm^\star)^\intercal$ can be interpreted geometrically: their directions coincide with the bisectors of the angle between $\mu$ and the horizontal axis. Equivalently, they split the set of admissible directions for $R_\pm$ (where the process remains recurrent) exactly in half. We immediately recover the orthogonality between $R_+^\star$ and $R_-^\star$ (which is consistent with $r_+^\star r_-^\star=-1$). See Figure~\ref{fig:phases} for an illustration.
\end{remark}

\begin{figure}
    \centering
    \includegraphics[width=0.75\linewidth]{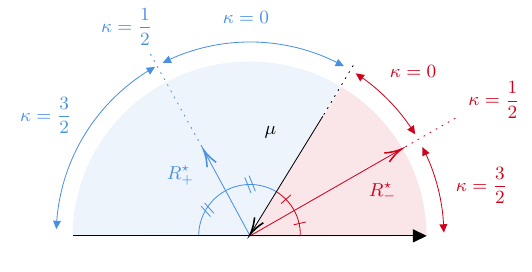}
    \caption{Geometric interpretation of the phase transitions. Blue and red exponents correspond to the asymptotics for $\pi_+$ and $\pi_-$, respectively. For a fixed drift $\mu$, the direction of $R_+^\star$ (\textit{resp.} $R_-^\star$) bisects the set of directions for $R_+$ (\textit{resp.} $R_-$) for which the process is recurrent, see Equation~\eqref{eq:rec}.}
    \label{fig:phases}
\end{figure}

Note that the asymptotics obtained in this corollary belong to the class of tail behaviors identified by Dai and Miyazawa in the quarter plane case, see~\cite{Dai_Miyazawa_2011}. We also point out that the model treated in Proposition~\ref{prop:ex} corresponds to a ``doubly critical'' case, in the sense that $r_+=r_+^\star$ and $r_-=r_-^\star$.

\section*{Acknowledgements}
The author is grateful to Sandro Franceschi for numerous insightful conversations and helpful suggestions throughout the development of this project.

\bibliographystyle{amsplain}
\bibliography{references}
\end{document}